\documentclass[11pt]{article}
\usepackage[centertags,intlimits]{amsmath}
\usepackage{amsfonts}
\usepackage{amsthm}
\usepackage{graphics}
\usepackage{color}
\usepackage{graphicx}

\allowdisplaybreaks[2]
\numberwithin{equation}{section}
\newtheorem{theorem}{Theorem}[section]
\newtheorem{lemma}{Lemma}[section]

\newtheorem{remark}{Remark}[section]

\newcounter{neweqn}
\newcommand{\beq}[1]{\addtocounter{neweqn}{1}\begin{equation}\label{#1}}
\newcommand{\eeq}{\end{equation}}

\newcommand{\nc}{\newcommand}

\newcommand{\beray}[1]{\addtocounter{neweqn}{1}\begin{eqnarray}  \label{#1}}

\nc{\vb}{\mathbf{v}}
\nc{\bx}{\mathbf{x}}
\nc{\by}{\mathbf{y}}
\nc{\bz}{\mathbf{z}}
\nc{\bu}{\mathbf{u}}
\nc{\bv}{\mathbf{v}}
\nc{\ba}{\mathbf{a}}
\nc{\bs}{\mathbf{s}}
\nc{\bq}{\mathbf{q}}
\nc{\bd}{\mathbf{d}}
\nc{\bb}{\mathbf{b}}
\nc{\bc}{\mathbf{c}}
\nc{\bi}{\mathbf{i}}
\nc{\bfr}{\mathbf{r}}
\nc{\bP}{\mathbf{P}}
\nc{\bQ}{\mathbf{Q}}
\nc{\R}{\mathbb R}
\nc{\N}{\mathbb N}
\nc{\bbC}{\mathbb C}
\nc{\D}{\mathbb D}
\nc{\Z}{\mathbb Z}
\nc{\F}{\mathbf F}
\nc{\bbS}{\mathbb S}
\nc{\bE}{\mathbf E}
\nc{\B}{\cal B}
\nc{\br}{\bigr}
\nc{\bl}{\bigl}
\nc{\Bl}{\Bigl}
\nc{\Br}{\Bigr}
\nc{\ind}[1]{\,\mathbf{1}_{\{#1\}}\,}
\author{A. Puhalskii}
\title{On tightness and exponential tightness 
\\ in generalised Jackson networks}

\begin{document}

\maketitle
\sloppy

\begin{abstract}
We give uniform proofs of  tightness and exponential tightness of
the sequences of
stationary queue lengths in ergodic generalised Jackson networks in a
number of  setups which concern  large, normal and moderate
deviations.
\end{abstract}

\section{Introduction}
\label{sec:introduction}

Both in weak convergence theory and large deviation theory,
 proofs of the convergence
 of the stationary distributions of stochastic
processes that converge trajectorially
often require establishing either tightness (for weak convergence)
 or exponential 
tightness (for large deviation convergence)
 of the stationary distributions. We
 formulate a uniform framework for proving 
 tightness properties of stationary queue lengths in generalised Jackson
 networks. 

We consider   standard generalised Jackson networks.
More specifically, a generic network consists of $K$ single server stations
and has
 a homogeneous customer population.
Customers arrive exogenously at the stations and are served  in the
order of arrival, one customer at a time. Upon being served, they either 
join a queue at another station or leave the network.
Let $A_k(t)$  denote the   
cumulative number of exogenous arrivals  at station $k$ 
by time $t$\,,
let $S_k(t)$ denote the cumulative number of customers   
that are served   
at station $k$ for the first $t$ units of busy time of that
station, and let 
$\Phi_{kl}(m)$
denote the cumulative number of customers among the   
first $m$ customers  departing station $k$ that go directly to station   
$l$. Let
$A_k=(A_k(t),\,t\in\R_+)$,
$S_k=(S_k(t),\,t\in\R_+)$, and $\Phi_k=(\Phi_{k}(m),\,m\in\Z_+)$,
where $\Phi_{k}(m)=(\Phi_{kl}(m),\,l\in\mathcal{K})$ and
 $\mathcal{K}=\{1,2,\ldots, K\}$\,. 
It is assumed that the $A_k$ and $S_k$ are, possibly delayed,
 nonzero renewal processes
and the customer routing is Bernoulli so that
$  \Phi_{kl}(m)=\sum_{i=1}^m 1_{\{\zeta_{k}^{(i)}=l\}}$,
where $\{\zeta_k^{(1)},\zeta_k^{(2)},\ldots\}$ is  a sequence of
i.i.d. r.v. assuming values in  $\mathcal{K}\cup\{0\}$\,,
$1_\Gamma$ standing for the indicator function of set $\Gamma$\,.
Let $Q_{k}(t)$ represent the number of customers  at station $k$ at time   
$t$\,.  The random entities $A_k$\,, $S_l$\,, $\Phi_i$ and $Q_j(0)$ are
assumed
to be defined on a common probability
space $(\Omega,\mathcal{F},\mathbf{P})$ and
be mutually independent,  where 
$k,l,i,j\in\mathcal{K}$\,.
We  denote $p_{kl}=\mathbf P(\zeta_k^{(1)}=l)$ and let $P=(p_{kl})_{k,l=1}^K$\,.
The matrix $P$ is assumed to be 
of spectral radius less than unity so
that every arriving customer eventually leaves.
All  the stochastic processes are assumed  to have
piecewise constant right--continuous with
left--hand limits trajectories. Accordingly, they are considered as random
elements of the associated Skorohod spaces.

Given $k\in\mathcal{K}$   and $t\in\R_+$, 
the following equations hold:
\begin{equation}
  \label{eq:2}
Q_k(t)=Q_k(0)+A_k(t)+\sum_{l=1}^K\Phi_{lk}\bl(D_{l}(t)\br)-D_k(t),
\end{equation}
where
\begin{equation}
  \label{eq:1}
D_k(t)=S_k\bl(B_k(t)\br)
\end{equation}
represents  the number of  departures  from station $k$ by time $t$ and
\begin{equation}
  \label{eq:5}
  B_k(t)=\int_0^t1_{\{Q_k(s)>0\}}\,ds
\end{equation}
represents the cumulative busy time of station $k$ by time $t$\,.
The process
$Q(t)=(Q_1(t),\ldots, Q_K(t))$ is not Markov, generally speaking, so
$Q(t)$ is  often appended with the backward recurrence times of the
exogenous arrival
processes and with the residual service times of customers in service.
The resulting process, $X(t)$\,,  is homogeneous Markov. It is then sensible to talk of
initial conditions.

Let, for 
$k\in\mathcal{K}$\,, nonnegative random variables $\xi_k$ and
$\eta_k$ represent generic times between exogenous arrivals  and
service times at station $k$, respectively. We assume that
$0<\mathbf E\xi_k<\infty$ and $0<\mathbf E\eta_k<\infty$\,.
 Let $\lambda_k=1/\mathbf E\xi_k$\,,
 $\mu_k=1/\mathbf E\eta_k$\,, $\lambda=(\lambda_1,\ldots,\lambda_K)^T$ and 
$\mu=(\mu_1,\ldots,\mu_K)^T$\,.
The network is said to be subcritical  (or to be normally loaded)
 if $\mu>(I-P^T)^{-1}\lambda$\,,
vector inequalities being understood entrywise.
The network is said to be
 in critical loading (also referred to as heavy traffic)
if $\mu=(I-P^T)^{-1}\lambda$\,. Under certain, fairly mild, hypotheses on the
generic interarrival times such as being unbounded and spread out
the process $X(t)$ in a subcritical network is positive Harris recurrent. Hence, there exists a
unique limit in
distribution of $X(t)$, as $t\to \infty$\,,
 no matter the initial
condition, the limit process being stationary,  see 
Dai \cite{Dai95}, Asmussen \cite{Asm87}.

There are three kinds of trajectorial asymptotics
 for the
process $Q(t)$\,.
They concern large, normal and moderate deviations.
When studying large deviations, the 
primitives of the network such as interarrival and service time CDFs are assumed
fixed and a large deviation principle (LDP) is obtained, as $n\to\infty$\,,
for the process $Q(nt)/n$ considered 
as a random element of the associated Skorohod space,
see Atar and Dupuis \cite{MR1719274},
Ignatiouk-Robert \cite{Ign00}, Puhalskii \cite{Puh07}.
The limit theorems on normal and moderate deviations
  concern   sequences of
networks indexed by  $n$ 
so that each piece of notation introduced
earlier is supplied with an additional subscript $n$\,.
It is assumed  that $\lambda_n\to\lambda$ and $\mu_n\to\mu$ with
$\mu=(I-P^T)^{-1}\lambda$\,, both $\lambda$ and $\mu$ being entrywise
positive.
 The number of stations, $K$\,, as well as the
routing decisions, $\Phi$, do not vary with $n$\,.
In the normal deviation limit theorem (also referred to as a diffusion
limit theorem) it is assumed also that the following critical loading
condition holds: as $n\to\infty$\,,
\begin{equation}
  \label{eq:49}
  \sqrt{n}(\lambda_n-(I-P^T)\mu_n)\to r\in\R^K\,.
\end{equation}
If the second moments of the service and interarrival times as well as
the initial conditions converge appropriately, then the scaled process
$Q_{n,k}(nt)/\sqrt n\,, k\in\mathcal{K}\,,$ converges 
in distribution in the associated Skorohod space
to a reflected $K$-dimensional diffusion, see Reiman \cite{Rei84}.
The moderate deviation limit theorem is concerned with
 the critical loading condition  of the form
\begin{equation}
  \label{eq:50}
  \frac{\sqrt n}{b_n}(\lambda_n-(I-P^T)\mu_n)\to r\in\R^K\,, 
\end{equation}
where $b_n$ is a numerical sequence such that $b_n\to \infty $ and $b_n/\sqrt
n\to0$\,.
 Under similar hypotheses, the process $1/(b_n\sqrt{
   n})Q_{n,k}(nt)\,, k\in\mathcal{K}\,,$
obeys an LDP for rate $b_n^2$ with a quadratic deviation function, Puhalskii \cite{Puh99a}.
It is plausible that in all three setups  the stationary distributions
of the processes in question, if well defined, should also converge
  appropriately.
Until recently, this sort of result was only available
for the waiting time process in the single server
queue, see Prohorov \cite{Pro63}  for normal deviations and
 Puhalskii \cite{Puh99a} for moderate deviations in critical loading.
It is a fairly straightforward consequence of tight (respectively,
exponentially tight) sequences of
probability measures being weakly (respectively, large deviation)
relatively compact that in order to go from the trajectorial
convergence to the convergence of the stationary distributions the following
 tightness properties are needed for the stationary queue
lengths,
\begin{enumerate}
\item for large deviations,
  \begin{equation}
    \label{eq:26}
    \lim_{x\to\infty}\limsup_{n\to\infty}\mathbf
    P(\frac{Q_k(0)}{n}\ge x)^{1/n}=0\,,\,k\in\mathcal{K},
  \end{equation}
\item for normal deviations,
  \begin{equation}
    \label{eq:27}
    \lim_{x\to\infty}\limsup_{n\to\infty}
\mathbf P(\frac{Q_{n,k}(0)}{\sqrt n}\ge x)=0\,,k\in\mathcal{K},
  \end{equation}
\item for moderate deviations,
  \begin{equation}
    \label{eq:28}
    \lim_{x\to\infty}\limsup_{n\to\infty}
\mathbf P(\frac{Q_{n,k}(0)}{b_n\sqrt
  n}\ge x)^{1/b_n^2}=0\,\,,k\in\mathcal{K}.
  \end{equation}
\end{enumerate}
As alluded to above, proof of the tightness asserted in \eqref{eq:27}
is a recent development, see Gamarnik and Zeevi \cite{GamZee06} and
Budhiraja and Lee \cite{LeeBud09}. 
Gamarnik and Zeevi \cite{GamZee06} used  Lyapunov
function techniques and relied on strong approximation of queueing processes
with diffusion  processes.
 Their hypotheses required the existence of certain 
conditional exponential moments of interarrival and service times.
Budhiraja and Lee \cite{LeeBud09} relaxed the moment conditions by
requiring uniform integrability of the squared 
interarrival and service times only. 
Lyapunov functions also played an important role.

   In this note, we 
provide direct  proofs to all three convergences in a uniform fashion.
It is done by building on the insight of Harrison and Reiman \cite{HarRei81}
that the queue lengh process, which is usually considered as an oblique
reflection process, can be viewed as  normal reflection of a
process that itself involves  the queue length  so that the
standard explicit 
formula for
the one-dimensional reflection can  be used.
This approach enables us to obtain an upper bound on the queue length
that is defined in terms of certain averages of the primitive
processes. Then, the ingenious device due to Prohorov \cite{Pro63}
for dealing with the suprema of  processes with negative linear drifts
over  infinite time intervals is brought to bear on the problem.
It is noteworthy that for normal deviations both our convergence and
 moment conditions are 
somewhat  weaker than in Budhiraja and Lee \cite{LeeBud09}.

\section{Main results}
\label{sec:main-results-1}

\begin{theorem}
  \label{the:tight}
  \begin{enumerate}
  \item 
 Suppose that a subcritical generalised Jackson network is stationary.  If 
$\mathbf Ee^{\theta\xi_k}<\infty$ and $\mathbf Ee^{\theta \eta_k}<\infty$ for some
$\theta>0$ and all $k\in\mathcal{K}$\,, then 
\eqref{eq:26} holds.
\item Suppose, for a sequence of subcritical stationary
generalised Jackson networks indexed by
  $n$\,,
$\lambda_n\to\lambda\in\R^K$\,, $\mu_n\to\mu\in\R^K$\,, both $\lambda$
and $\mu$ being entrywise positive,
and \eqref{eq:49} holds, with $r$ having negative
entries.
 If $\sup_n\mathbf E\xi_{n,k}^2<\infty$
and $\sup_n\mathbf E\eta_{n,k}^2<\infty$\,, for all
$k\in\mathcal{K}$\,,
 then \eqref{eq:27} holds.
\item
Suppose, for a sequence of subcritical stationary
generalised Jackson networks indexed by
  $n$\,,
$\lambda_n\to\lambda\in\R^K$\,, $\mu_n\to\mu\in\R^K$\,, both $\lambda$
and $\mu$ being entrywise positive,
and \eqref{eq:50} holds, with $r$ having negative
entries, where $b_n\to\infty$ and $b_n/\sqrt n\to0$\,.
 Let either of the following conditions hold:
 \begin{enumerate}
 \item 
 for some $\epsilon>0$ and all $k\in\mathcal{K}$\,, $\sup_n\mathbf E\xi_{n,k}^{2+\epsilon}<\infty$
and $\sup_n\mathbf E\eta_{n,k}^{2+\epsilon}<\infty$\,, and $\sqrt{\ln n}/b_n\to\infty$\,, 
\item for some $\alpha>0$\,, $\beta\in(0,1]$ and all $k\in\mathcal{K}$\,, 
$\sup_n\mathbf E\exp(\alpha \xi_{n,k}^\beta)<\infty$ and  
$\sup_n\mathbf E\exp(\alpha \eta_{n,k}^\beta)<\infty$\,, and 
$n^{\beta/2}/b_n^{2-\beta}\to\infty$\,.
 \end{enumerate}
  \end{enumerate}
Then \eqref{eq:28} holds.
\end{theorem}

\section{Auxiliary results}
\label{sec:aux-results}
In this section we develop
  upper bounds on the queueuing processes.
A subcritical network is assumed to be started in a stationary state
meaning that the process $X(t)$ is stationary.
 Then, the processes $Q_k(t)$ are
stationary and the processes $A_k(t)$\,,  $B_k(t)$\,, $D_k(t)$ 
and $S_k(t)$ have
stationary increments with $A_k(0)=D_k(0)=B_k(0)=S_k(0)=0$\,. 
 All the processes  are extended to processes on the whole
real line with the same finite-dimensional distributions. The processes
$A_k(t)$\,, $D_k(t)$\,, $B_k(t)$ and $S_k(t)$ thus assume
negative values  on the negative halfline. 
The basic equations in \eqref{eq:2}, \eqref{eq:1}, and \eqref{eq:5}
still hold for $t<0$\,.

Let us  introduce  ''centred'' versions 
of the primitive processes
by
\[
\overline A_k(t)=A_k(t)-\lambda_kt\,,\quad \overline S_k(t)=S_k(t)-\mu_kt\,,\quad
\overline\Phi_{lk}(m)=\Phi_{lk}(m)-p_{lk}m\,.
\]
Let
\[
  \overline D_k(t)=\overline S_k(B_k(t))
\]
and \begin{equation}
  \label{eq:8}
    \overline{Q}_k(t)=
Q_k(0)+\overline A_k(t)+\sum_{l=1}^K\overline \Phi_{lk}\bl(D_{l}(t)\br)
+\sum_{l=1}^Kp_{lk}\overline{D_l}(t)-\overline{D}_k(t)\,.
\end{equation}

By \eqref{eq:2} and \eqref{eq:1}, 
\[
Q_k(t)
=\overline{Q}_k(t)-\nu_kt
-\sum_{l=1}^Kp_{lk}\mu_l(t-B_l(t))+\mu_k(t-B_k(t))\,,
\]
where \[
  \nu_k=-\lambda_k-\sum_{l=1}^Kp_{lk}\mu_l+\mu_k\,.
\]
In vector form, with
$Q(t)=(Q_k(t)\,,k\in\mathcal{K})^T$\,, $\overline Q(t)=(\overline 
Q_k(t)\,,k\in\mathcal{K})^T$\,,
$\varphi(t)=(\varphi_k(t)\,,k\in\mathcal{K})^T$\,,
$\nu=(\nu_k\,,k\in\mathcal{K})^T$\,, and $I$ representing
the $K\times K$--identity matrix,
\begin{equation}
  \label{eq:9}
  Q(t)=\overline{Q}(t)-\nu t+(I-P^T)\varphi(t)\,,
\end{equation}
where 
$\varphi_k(t)=\mu_k(t-B_k(t))$\,.
For each $k$\,, $Q_k(t)$ is seen to be the  reflection of the
$k$--th component of $\overline{Q}(t)-\nu t-P^T\varphi(t)$\,.
By the properties of the one--dimensional reflection and some algebra,
\[
  \varphi(t)=-\inf_{s\le t}(\overline{Q}(s)-\nu
  s-P^T\varphi(s))\wedge0
\le -\inf_{s\le t}(\overline{Q}(s)-\nu s)\wedge0+P^T\varphi(t)\,.
\]
Solving the latter inequality for $\varphi(t)$ obtains, accounting for
the matrix $(I-P^T)^{-1}$ being nonnegative,
\begin{equation}
  \label{eq:17}
  \varphi(t)\le
-\inf_{s\le t}((I-P^T)^{-1}\overline{Q}(s)-\hat\nu s)\wedge0\,,
\end{equation}
where 
$  \hat\nu=(I-P^T)^{-1}\nu\,.$
By the network being subcritical, $\hat\nu>0$ entrywise. For economy
of notation, we introduce
$  \hat P=(I-P^T)^{-1}\,.$
Let
\[
  \hat Q(t)=\hat PQ(t)\,.
\]
By the Neumann series,
\begin{equation}
  \label{eq:16}
  \hat Q(t)\ge Q(t)\,.
\end{equation}
  Multiplying \eqref{eq:9} through
 with $\hat P$ 
and using
 \eqref{eq:17} 
yield
\begin{equation}
  \label{eq:15}
  \hat  Q(t)
\le \hat P\overline{Q}(t)
-\hat\nu t+
\sup_{s\le
  t}(\hat\nu s-\hat P\overline{Q}(s)
)\vee0\,.
\end{equation}
Since the network is stationary, we obtain from \eqref{eq:15},
via a left time shift by $t$ and a change of variables, that
 for $u\in\R_+^K$\,,
\[  \mathbf P(\hat  Q(0)\ge u)\le
\mathbf P(\sup_{0\le s\le t  }(\hat P\overline Q(0)-\hat P\overline Q(-s)
-\hat\nu s)\vee (\hat P\overline Q(0)
-\hat\nu t) \ge u)\]
so that, on letting $t\to\infty$\,,
\begin{equation}
  \label{eq:23}
  \mathbf P(\hat  Q(0)\ge u)\le
\mathbf P(\sup_{s\ge0 }(\hat P\overline Q(0)-\hat P\overline Q(-s)
-\hat\nu s) \ge u)\,.
\end{equation}
Let $ \hat A(s)=\hat P \overline A(s)$ and $\hat \Phi_l(m)=
\hat P\overline \Phi_{l}(m)$\,, where
$\overline \Phi_{l}(m)=(\overline \Phi_{lk}(m)\,, k\in\mathcal{K})^T$\,.
 Let us also introduce
   $\overline D(s)=(\overline D_k(s)\,,k\in\mathcal{K})^T$\,.
Owing to \eqref{eq:8},
\begin{equation}
  \label{eq:25}
  \hat P\overline{Q}(0)-\hat P\overline{Q}(-s)=
-\hat A(-s)-\sum_{l=1}^K\hat \Phi_{l}\bl(D_{l}(-s)\br)
+\overline D(-s)\,.
\end{equation}
Let  $\tau_{k,i}$  (respectively, $\sigma_{l,i}$)
represent the $i$-th arrival time of $A_k(t)$ (respectively, of
                                                 $S_l(t)$)\,.
Let $\hat r$ represent the Perron--Frobenius eigenvalue of $\hat P$\,.
(We note that $\hat r\ge1$\,.) Let $v=(v_k\,,k\in\mathcal{K})$ 
represent an associated with  $\hat r$ left
eigenvector, which is a  row  $K$--vector with
nonnegative entries.
In the next lemma, we use the notation
 $\hat u=\max_{k\in\mathcal{K}:\,v_k>0}u_k$\,, $\hat v=
\max_{k\in\mathcal{K}}v_k$ and $\breve v=\min_{k\in\mathcal{K}:\,v_k>0}v_k$\,.
 \begin{lemma}
   \label{le:hatbound}
For $u=(u_k\,,k\in\mathcal{K})^T\in\R_+^K$ and
$\alpha=(\alpha_k\,,k\in\mathcal{K})^T\in\R_+^K$\,,
\begin{equation}
  \label{eq:4}
  \mathbf P(
\sup_{s\ge0}(-\hat A(-s)-\alpha s)> u)\le
\sum_{{k}=1}^K
\mathbf P(
\sup_{i\in\N}(i-\lambda_{k} \tau_{{k},i}-\frac{\alpha_k\tau_{{k},i}}{\hat r})>
\frac{\breve v\hat u}{\hat v\hat r}))\,,
\end{equation}
\begin{equation}  \label{eq:10}
\mathbf P(  \sup_{s\ge0}(-\hat \Phi_l(-D_l(-s))-\alpha s)> u)\le
\sum_{k=1}^K\mathbf P(  \sup_{i\in\N}(- \overline
 \Phi_{lk}(i)-\frac{\alpha_k \sigma_{l,i}}{\hat r})> 
\frac{\breve v\hat u}{\hat v\hat r})\,,
\end{equation}
\begin{equation}  \label{eq:11}
\mathbf P(  \sup_{s\ge0}(\overline D(-s)-\alpha s)> u)\le
\min_{k\in\mathcal{K}}\mathbf P(  \sup_{i\in\N}(-i+1+\mu_k\sigma_{k,i}-\alpha_k 
\sigma_{k,i})>u_k)\,.
\end{equation}
 \end{lemma}
 \begin{proof}
 Noting that the process
$(-A(-s)\,, s\ge0)$ is distributed as the process
 $(A(s)\,,s\ge0)$\,, 
both being equilibrium renewal processes with the same generic
interarrival time distributions, we have that  
\begin{multline*}
      \mathbf P(
\sup_{s\ge0}(\hat A(s)-\alpha s)> u)
=  \mathbf P(
\sup_{s\ge0}(\hat P\overline A(s)-\alpha s)> u )\\
\le  \mathbf P(
\sup_{s\ge0}(\hat rv\overline A(s)-v\alpha s)> vu )
\le\sum_{k=1 }^K\mathbf P(
\sup_{s\ge0}(\hat rv_{k}\overline A_{k}(s)-v_k\alpha_k s)>\breve v\hat u)
\\\le\sum_{k=1}^K\mathbf P(
\sup_{s\ge0}(\overline A_{k}(s)-\frac{\alpha_k s}{\hat r})>
\frac{\breve v\hat u}{\hat v\hat r})=\sum_{{k}=1}^K
\mathbf P(
\sup_{i\in\N}(i-\lambda_{k} \tau_{{k},i}-\frac{\alpha_k\tau_{{k},i}}{\hat r})>
\frac{\breve v\hat u}{\hat v\hat r})
\,,
\end{multline*}
which proves \eqref{eq:4}. The latter equality holds because the $\sup$
over $s$ can be taken over the jump times of $\overline A(s)$\,.
For \eqref{eq:10}, we write, taking into account that $0\ge B_l(-s)\ge
-s$ and that the process $(-S_l(-s)\,,s\ge0)$ is distributed as 
$(S_l(s)\,,s\ge0)$\,,
\begin{multline*}
\mathbf P(  \sup_{s\ge0}(-\hat\Phi_l(-D_l(-s))-\alpha s)> u)=
\mathbf P(\sup_{s\ge0}(-\hat \Phi_l(-S_l(B_l(-s)))-\alpha
s))> u)\\
\le
\mathbf P(  \sup_{s\ge0}(-(\hat \Phi_l(-S_l(B_l(-s)))
+\alpha\,B_l(- s))> u))
\\\le\mathbf P(  \sup_{s\ge0}(-\hat P\overline \Phi_l(-S_l(-s))-\alpha
s)> u)\\
\le\mathbf P(  \sup_{s\ge0}(-\hat r v
\overline \Phi_l(-S_l(-s))-v\alpha
s)> v u) \\
\le\sum_{k=1}^K\mathbf P(  \sup_{s\ge0}(-\hat r
\overline \Phi_{lk}(-S_l(-s))-\alpha_k
s)>\frac{\breve v\hat u}{\hat v})
\\=\sum_{k=1}^K\mathbf P(  \sup_{s\ge0}(-\hat r
\overline \Phi_{lk}(S_l(s))-\alpha_k
s)>\frac{\breve v\hat u}{\hat v})
=\sum_{k=1}^K\mathbf P(  \sup_{i\in\N}(-\hat r
\overline \Phi_{lk}(i)-\alpha_k
\sigma_{l,i})> \frac{\breve v\hat u}{\hat v}))\,.
\end{multline*}
The proof of \eqref{eq:11} proceeds similarly. 
 \end{proof}

\section{Proof of the main results.}
In this section, Theorem \ref{the:tight} is proved.
  We let 
$\overline\tau_{k,i}=\tau_{k,i}-\tau_{k,1}-
\mathbf E (\tau_{k,i}- \tau_{k,1})$ (respectively,
$\overline\sigma_{l,i}=\sigma_{l,i}-\sigma_{l,1}-
\mathbf E (\sigma_{l,i}-\mathbf  \sigma_{l,1}) $) \,.
It is noteworthy that, if $i\ge2$\,, then 
 $\overline\tau_{k,i}$ (respectively,
$\overline\sigma_{l,i}$) 
is the sum of $(i-1)$ zero mean i.i.d. r.v. 

\begin{proof}[Proof of part 1]
  By \eqref{eq:16}, \eqref{eq:23}\,, \eqref{eq:25}, 
Lemma \ref{le:hatbound},  and the
  fact that $\tau_{k,1}$ and $\sigma_{l,1}$ satisfy the Cram{\'e}r condition,
 it suffices
 to prove that, given arbitrary $C$ and
$\epsilon>0$\,, for all $k\in\mathcal{K}$ and
all $l\in\mathcal{K}$\,,
\begin{align*}
  \lim_{x\to\infty}\limsup_{n\to\infty}  \mathbf P(
\sup_{i\in\N}(-(\lambda_k+\frac{\epsilon\hat\nu_k}{\hat r})
\overline \tau_{k,i+1}
-\frac{\epsilon\hat\nu_k}{\lambda_k\hat r}i)
\ge
 nx)^{1/n}=0\,,
\\\notag
 \lim_{x\to\infty}\limsup_{n\to\infty}
\mathbf P\bl(
\sup_{i\in\N}
(-\overline \Phi_{lk}(i)
-
\frac{\epsilon\hat\nu_k}{\hat r}i)\ge nx\br)^{1/n}=0\,,
\\ \notag 
  \lim_{x\to\infty}\limsup_{n\to\infty}  \mathbf P(
\sup_{i\in\N}(
C\overline \sigma_{k,i+1}
-\frac{\epsilon\hat\nu_k}{
\hat r}i)
\ge
 nx)^{1/n}=0\,.
\end{align*}
We prove the first convergence, the other two being proved similarly.
It is sufficient to prove that, no matter number $\tilde C$\,,
\[
  \lim_{x\to\infty}\limsup_{n\to\infty}  \mathbf P(
\sup_{i\in\N}(\tilde C\overline \tau_{k,i+1}
-i)
\ge
 nx)^{1/n}=0\,.
\]
We note that, cf. Prohorov \cite{Pro63}, Puhalskii
\cite{Puh99a},
\begin{multline}
  \label{eq:22}
  \mathbf P(
\sup_{i\in\N}(\tilde C
\overline \tau_{k,i+1}
-i)\ge nx)\le
\mathbf P(
\max_{1\le i\le \lfloor nx\rfloor}\tilde C
\overline \tau_{k,i+1}\ge nx)\\+
\sum_{j=\lfloor\log_2\lfloor nx\rfloor\rfloor}^{\infty}\mathbf P(
\max_{2^j+1\le i\le 2^{j+1}}
(\tilde C
\overline \tau_{k,i+1}
-i)>0)
\\\le \mathbf P(
\max_{1\le i\le \lfloor nx\rfloor}\tilde C
\overline \tau_{k,i+1}\ge nx)+
\sum_{j=\lfloor\log_2\lfloor nx\rfloor\rfloor}^{\infty}\mathbf P(
\tilde C
\overline \tau_{k,2^j+1}
-\,2^{j-1}>0)\\
+ \sum_{j=\lfloor\log_2\lfloor nx\rfloor\rfloor}^{\infty}\mathbf P(
\max_{2^j+1\le i\le 2^{j+1}}
(\tilde C
(\overline \tau_{k,i+1}-\overline \tau_{k,2^j+1})
-\,(i-2^{j-1})
)>0)\\
\le
\mathbf P(
\max_{1\le i\le \lfloor nx\rfloor}\tilde C
\overline \tau_{k,i+1}\ge nx)
+2 \sum_{j=\lfloor\log_2\lfloor nx\rfloor\rfloor}^{\infty}\mathbf P(
\max_{1\le i\le 2^{j}}
\tilde C
\overline \tau_{k,i+1}\ge 
\,2^{j-1})\,.
\end{multline}
By Doob's inequality, for $\vartheta>0$\,,
 with $j=\lfloor\log_2\lfloor nx\rfloor\rfloor+m$\,,
 \begin{multline}
   \label{eq:3}
     \mathbf P(
\max_{1\le i\le 2^{j}}
 \tilde C
\overline \tau_{k,i+1}\ge 
\,2^{j-1})
\le \mathbf P(
\max_{1\le i\le\lfloor nx\rfloor 2^{m}}
 \tilde C
\overline \tau_{k,i+1}\ge 
\,\lfloor nx\rfloor 2^{m-2})\\\le
  \frac{\mathbf Ee^{\vartheta \tilde C
\overline \tau_{k,\lfloor nx\rfloor 2^{m}+1}}}{
e^{\vartheta\lfloor nx\rfloor 2^{m-2}}}=
\bl(e^{-\vartheta/4}\mathbf Ee^{\vartheta \tilde C
\overline\xi_k}\br)^{\lfloor nx\rfloor2^m}\,,
\end{multline}
where $\overline\xi_k=\xi_k-\mathbf E\xi_k$\,.
Since $\mathbf E\overline\xi_k=0$\,, for
 small enough $\vartheta$\,, we have that
$e^{-\vartheta/4}\mathbf Ee^{\vartheta \tilde C
\overline\xi_k} <1$\,.
Hence, with $\varrho=e^{-\vartheta/4}\mathbf Ee^{\vartheta \tilde C
\overline\xi_k}$\,,  for great enough $n$ and small enough $\vartheta$\,,
\[
 \mathbf P(
\max_{1\le i\le 2^{j}}
 \tilde C
\overline \tau_{k,i+1}\ge 
\,2^{j-1})^{1/n}\le
 \varrho^{x2^{m-1}}
\]
so that
\[
  \sum_{j=\lfloor\log_2\lfloor nx\rfloor\rfloor}^{\infty}\mathbf P(
\max_{1\le i\le 2^{j}}
 \tilde C
\overline \tau_{k,i+1}\ge 
\,2^{j-1})^{1/n}\le
\sum_{m=0}^\infty\varrho^{ x2^{m-1}}\,.
\]
By dominated convergence,  the latter series tends to $0$\,, as
$x\to\infty$\,.

Also,
\[
  \mathbf P(
\max_{1\le i\le \lfloor nx\rfloor} \tilde C
\overline \tau_{k,i+1}\ge  nx)^{1/n}\le 
(\mathbf Ee^{\vartheta \tilde C
\overline \xi_k})^{\lfloor nx\rfloor/n}\,,
\]
which tends to zero, as $n\to\infty$ and $x\to\infty$ provided
$\vartheta$ is small enough.
\end{proof}
\begin{proof}[Proof of part 2]
      Noting that 
$\sup_n\mathbf E\tau_{n,k,1}<\infty$ and
$\sup_n\mathbf E\sigma_{n,l,1}<\infty$\,, we have that
it suffices to establish the following
analogues of the convergences in the proof of part 1,
\begin{align}
\notag
    \lim_{x\to\infty}\limsup_{n\to\infty}
  \mathbf P(\sup_{i\in\N}(-(\lambda_{n,k}+\frac{\epsilon\hat\nu_{n,k}}{\hat r})
\overline \tau_{n,k,i+1}
-\frac{\epsilon\hat\nu_{n,k}}{\lambda_{n,k}\hat r}i)
\ge
  \sqrt{n}x)=0\,,
\\\notag
 \lim_{x\to\infty}\limsup_{n\to\infty}
\mathbf P\bl(
\sup_{i\in\N}
(-\overline \Phi_{lk}(i)
-\frac{\epsilon\hat\nu_{n,k}}{\hat r}\,i)\ge \sqrt{n}x\br)=0\,,
\\  \label{eq:29}  \lim_{x\to\infty}\limsup_{n\to\infty}\mathbf P\bl(
\sup_{i\in\N}(
C_n\overline\sigma_{n,k,i+1}
-\frac{\epsilon\hat\nu_{n,k}}{\mu_{n,k}\hat r}
i)\ge \sqrt{n}x\br)=0\,,
\end{align}
where $C_n$ is a bounded sequence.
We  work on the third convergence in \eqref{eq:29}, the other two being proved
similarly.
It is sufficient to prove that 
\[
  \lim_{x\to\infty}\limsup_{n\to\infty}  \mathbf P(
\sup_{i\in\N}(C_n\overline \sigma_{n,k,i+1}
-\hat\nu_{n,k}i)
\ge
 \sqrt nx)=0\,.
\]
Reasoning as in \eqref{eq:22} yields
\begin{multline}
  \label{eq:31}
  \mathbf P(
\sup_{i\in\N}(C_n\overline\sigma_{n,k,i+1}
-\hat\nu_{n,k}i)\ge \sqrt{n}x)
\le
\mathbf P(
\max_{1\le i\le \lfloor nx\rfloor}C_n\overline\sigma_{n,k,i+1}\ge  \sqrt nx)\\
+2 \sum_{j=\lfloor\log_2\lfloor nx\rfloor\rfloor}^{\infty}\mathbf P(
\max_{1\le i\le 2^{j}}
C_n\overline\sigma_{n,k,i+1}\ge\hat\nu_{n,k}
\,2^{j-1})\,.
\end{multline}
By  Kolmogorov's inequality, in light of the definition
of $\overline\sigma_{n,k,i}$\,,
\[  \mathbf P(
  \max_{1\le i\le 2^{j}}
  C_n\overline\sigma_{n,k,i+1}\ge\hat \nu_{n,k}
2^{j-1})
\le
\frac{C_n^2\mathbf E\eta_{n,k}^2}{\hat\nu_{n,k}^22^{j-2}}
\,.
\]
As $\sum_{j=\lfloor\log_2\lfloor nx\rfloor\rfloor}^{\infty}2^{-j}\le
4/\lfloor
nx\rfloor$ and $\sqrt n\hat\nu_{n}\to-\hat Pr$\,, the series on the righthand side of \eqref{eq:31} tends
to $0$\,, as $n\to\infty$ and $x\to\infty$\,. By a similar calculation,
\[\lim_{x\to\infty}\limsup_{n\to\infty}\mathbf P\bl(
\max_{1\le i\le \lfloor nx\rfloor}C_n\overline\sigma_{n,k,i+1}\ge  \sqrt nx\br)
=0\,.\]

\end{proof}
\begin{proof}[Proof of part 3.]
It suffices to prove the following,
\begin{align*}
\notag
    \lim_{x\to\infty}\limsup_{n\to\infty}
  \mathbf P\bl(\sup_{i\in\N}(-(\lambda_{n,k}+\frac{\epsilon\hat\nu_{n,k}}{\hat r})
\overline \tau_{n,k,i+1}
-\frac{\epsilon\hat\nu_{n,k}}{\lambda_{n,k}\hat r}i)
\ge
  b_n\sqrt{n}x\br)^{1/b_n^2}=0\,,\notag
\\
 \lim_{x\to\infty}\limsup_{n\to\infty}
\mathbf P\bl(
\sup_{i\in\N}
(-\overline \Phi_{lk}(i) 
-\frac{\epsilon\hat\nu_{n,k}}{\hat r}\,i)\ge b_n\sqrt{n}x\br)^{1/b_n^2}=0\,,
\\   \lim_{x\to\infty}\limsup_{n\to\infty}\mathbf P\bl(
\sup_{i\in\N}(
\overline C_n\overline\sigma_{n,k,i+1}
-\frac{\epsilon\hat\nu_{n,k}}{\mu_{n,k}\hat r}
i)\ge b_n\sqrt{n}x\br)^{1/b_n^2}=0\,,\notag
\end{align*}
where $\overline C_n$ is a bounded sequence.
We  provide a proof of the second convergence above.
In analogy with \eqref{eq:22} and \eqref{eq:31},
\begin{multline}
  \label{eq:7}
   \mathbf P\bl(
\sup_{i\in\N}(-\overline \Phi_{lk}(i)
-\frac{\epsilon}{\hat r}\hat\nu_{n,k}i)\ge 
b_n \sqrt{n}x\br)\le
\mathbf P(\max_{1\le i\le \lfloor nx \rfloor}(-\overline \Phi_{lk}(i))\ge
 b_n\sqrt{n}x\br)\\
+2\sum_{j=\lfloor \log_2\lfloor nx\rfloor\rfloor}^\infty
\mathbf P(\max_{1\le i\le 2^j}(-\overline \Phi_{lk}(i))
\ge \frac{\epsilon}{\hat r}\hat\nu_{n,k}\,2^{j-1})
\,.
\end{multline}
On noting that $\sqrt{n}/b_n\,\hat\nu_{n,k}$ converges to the $k$th
entry of $-\hat Pr$\,,
the proof is concluded by an application of Lemma A.1 in
Puhalskii \cite{Puh99a}.  For 
instance, under the hypotheses of part (a), we can write
in view of part (i) of Lemma A.1 in
Puhalskii \cite{Puh99a} for a generic
term of the series 
on the righthand side of \eqref{eq:7} in analogy with  \eqref{eq:3},
for some $\breve C_1>0$ and $\breve C_2>0$\,,  and for all $n$ great enough,
\begin{multline*}
  \mathbf P(
\max_{1\le i\le 2^{j}}
 (-\overline \Phi_{lk}(i))\ge \frac{\epsilon}{\hat r}\,\hat\nu_{n,k}
\,2^{j-1})
\le e^{-\breve C_1b_n^2\sqrt{x}
    2^{m/2-1}}+\breve C_2\frac{b_n^{2+\epsilon}}{n^{\epsilon/2}}
\frac{1}{( x
    2^{m-3})^{\epsilon/2}}\,.
\end{multline*}
When raised to the power of $1/b_n^2$ the latter sum is bounded above
by
\[
e^{-\breve C_1\sqrt{x}
    2^{m/2-1}}+\breve C_2^{1/b_n^2}\frac{b_n^{(2+\epsilon)/b_n^2}}{n^{\epsilon/(2b_n^2)}}
\frac{1}{( x
    2^{m-3})^{\epsilon/(2b_n^2)}}\,.\]

It follows that, when raised to the power of $1/b_ n^2$\,, 
the series in question vanishes as $n\to\infty$ and $x\to\infty$\,. \end{proof}
\begin{remark}
  Theorem 2.2 in Puhalskii \cite{Puh19a} asserts an LDP for a
  stationary
subcritical
   generalised Jackson network. Unfortunately, I
misapplied
 Theorem 4.1 in Meyn and Down \cite{MeyDow94} by assuming that it
 concerned a standard generalised Jackson network.
Actually, the hypotheses of the theorem in 
question require that the arrival processes at the stations
be obtained by  splitting another counting renewal process, so, the
arrival processes are not independent, generally speaking. 
The exponential tightness in part 1 of Theorem 1 of this paper can
be used to give  a correct proof, see Puhalskii \cite{Puh26}. Similarly,
the assertion of part 3 of Theorem 1 paves the way for a proof that
the moderate deviations of the stationary queue lengths are governed by
the associated quasipotential.
\end{remark}


\begin{thebibliography}{00}
\def\cprime{$'$} \def\cprime{$'$} \def\cprime{$'$} \def\cprime{$'$}
  \def\cprime{$'$} \def\cprime{$'$} \def\cprime{$'$} \def\cprime{$'$}
  \def\cprime{$'$} \def\cprime{$'$} \def\cprime{$'$}
  \def\polhk#1{\setbox0=\hbox{#1}{\ooalign{\hidewidth
  \lower1.5ex\hbox{`}\hidewidth\crcr\unhbox0}}} \def\cprime{$'$}
  \def\cprime{$'$} \def\cprime{$'$}

\bibitem{Asm87}
S.~Asmussen.
\newblock Applied Probability and Queues.
     \newblock John Wiley \& Sons, 1987.

\bibitem{MR1719274}
R.~Atar and P.~Dupuis.
\newblock Large deviations and queueing networks: methods for rate function
  identification.
\newblock {\em Stoch. Proc. Appl.}, 84(2):255--296, 1999.

\bibitem{LeeBud09}
A.~Budhiraja and C.~Lee.
\newblock Stationary distribution convergence for generalized {J}ackson
  networks in heavy traffic.
\newblock {\em Math. Oper. Res.}, 34(1):45--56, 2009.

\bibitem{Dai95}
J.~Dai.
\newblock
On positive Harris recurrence of multiclass queueing networks: a unified approach via fluid limit models.
\newblock{\em Ann.  Appl. Probab.}, 5(1): 49--77, 1995.


\bibitem{GamZee06}
D.~Gamarnik and A.~Zeevi.
\newblock Validity of heavy traffic steady-state approximations in generalized
  {J}ackson networks.
\newblock {\em Ann.  Appl. Probab.}, 16(1):56--90, 2006.

\bibitem{HarRei81}
J.~Harrison and M.~Reiman.
\newblock Reflected {B}rownian motion on an orthant.
\newblock{\em Ann.  Probab.}, 9(2):302--308, 1981.

\bibitem{Ign00}
I.~Ignatiouk-Robert.
\newblock Large deviations of {J}ackson networks.
\newblock {\em Ann. Appl. Probab.}, 10(3):962--1001, 2000.

\bibitem{MeyDow94}
S.~Meyn and D.~Down.
\newblock Stability of generalized {J}ackson networks.
\newblock {\em Ann. Appl. Probab.}, 4(1):124--148, 1994.

\bibitem{Pro63}
{Yu.}~Prohorov.
\newblock Transient phenomena in queueing processes.
\newblock {\em Lit. Mat. Rink.}, 3:199--206, 1963.
\newblock (in Russian).

\bibitem{Puh95}
A.~Puhalskii.
\newblock Large deviation analysis of the single server queue.
\newblock {\em Queueing Syst. Th. Appl.}, 21(1-2):5--66, 1995.

\bibitem{Puh99a}
A.~Puhalskii.
\newblock Moderate deviations for queues in critical loading.
\newblock {\em Queueing Syst. Th. Appl.}, 31(3-4):359--392, 1999.


\bibitem{Puh07}
A.~Puhalskii.
\newblock The action functional for the {J}ackson network.
\newblock {\em Markov Proc. Rel. Fields}, 13:99--136, 2007.

\bibitem{Puh19a}
A.~Puhalskii.
\newblock Large deviations of the long term distribution of a non {M}arkov
  process.
\newblock {\em Electron. Commun. Probab.}, 24:Paper No. 35, 11, 2019.
\bibitem{Puh26}
A.~Puhalskii.
\newblock On large queue lengths in generalised Jackson networks.
\newblock {\em Markov Proc. Rel. Fields}, 2026 (submitted).

\bibitem{Rei84}
M.~Reiman.
\newblock Open queueing networks in heavy traffic.
\newblock {\em Math. Oper. Res.}, 9(3):441--458, 1984.

\end{thebibliography}
\end{document}